\documentclass{custom_class} 

\title[Complexity of finding stationary points in one dimension]{On the complexity of finding stationary points of smooth functions in one dimension}
\usepackage{amsmath}
\usepackage{bbm}
\usepackage{booktabs}
\usepackage{changepage}
\usepackage{enumitem}
\usepackage{euscript}
\usepackage{graphicx}
\usepackage{mathrsfs}
\usepackage{mathtools}
\usepackage{mdframed}
\usepackage{microtype}
\usepackage{tikz}
\usepackage{times}
\usepackage[Symbolsmallscale]{upgreek}
\usepackage{thm-restate} 

\usepackage{custom}

\makeatletter
\def\set@curr@file#1{\def\@curr@file{#1}} 
\makeatother
\usepackage[load-configurations=version-1]{siunitx} 

\usepackage{float}

\coltauthor{\Name{Sinho Chewi}\thanks{This work was completed while SC was a research intern at Microsoft Research.} \Email{schewi@mit.edu}\\
\addr Massachusetts Institute of Technology
\AND
 \Name{S\'ebastien Bubeck} \Email{sebubeck@microsoft.com} \\
 \Name{Adil Salim} \Email{adilsalim@microsoft.com}\\
 \addr Microsoft Research
}

\begin{document}

\maketitle

\begin{abstract}
    We characterize the query complexity of finding stationary points of one-dimensional non-convex but smooth functions.
    We consider four settings, based on whether the algorithms under consideration are deterministic or randomized, and whether the oracle outputs $1^{\rm st}$-order or both $0^{\rm th}$- and $1^{\rm st}$-order information.
   Our results show that algorithms for this task provably benefit by incorporating either randomness or $0^{\rm th}$-order information. Our results also show that, for every dimension $d \geq 1$, gradient descent is optimal among deterministic algorithms using $1^{\rm st}$-order queries only.
\end{abstract}




\section{Introduction}

We consider optimizing a non-convex but smooth function $f : \R^d\to\R$, a task which underlies the spectacular successes of modern machine learning.
Despite the fundamental nature of this question, there are still important aspects which remain poorly understood.

To set the stage for our investigation, let $f : \R^d\to\R$ be a $\beta$-smooth function with bounded objective gap: $f(0) - \inf f \le \Delta$.
Since global minimization of $f$ is, in general, computationally intractable~\citep[c.f.][]{nemirovskyyudin1983complexity}, we focus on the task of outputting an \emph{$\varepsilon$-stationary point}, that is, a point $x^\star\in\R^d$ such that $\norm{\nabla f(x^\star)} < \varepsilon$.
By a standard rescaling argument (see Lemma~\ref{lem:reduction}), it suffices to consider the case $\beta = \Delta = 1$.
Then, it is well-known~\citep[see, e.g.,][]{nesterov2018cvxopt}, that the standard gradient descent (GD) algorithm solves this task in $O(1/\varepsilon^2)$ queries to an oracle for the gradient $\nabla f$.
Conversely,~\cite{carmonetal2020stationarypt} proved that if the dimension $d$ is sufficiently large, then any randomized algorithm for this task must use at least $\Omega(1/\varepsilon^2)$ queries to a local oracle for $f$, thereby establishing the optimality of GD in high dimension.

However, the \emph{low-dimensional complexity} of computing stationary points remains open.
Indeed, the main limitation of~\citet{carmonetal2020stationarypt} is that their lower bound constructions require the ambient dimension to be large: more precisely, they require $d \ge \Omega(1/\varepsilon^2)$ for deterministic algorithms, and $d \ge \widetilde \Omega(1/\varepsilon^4)$ for randomized algorithms.
The large dimensionality arises because they adapt to the non-convex smooth setting a ``chain-like'' lower bound construction for optimization of a convex non-smooth function~\citep{nesterov2018cvxopt}.
The chain-like construction forces certain natural classes of iterative algorithms to explore only one new dimension per iteration, and hence the dimension of the ``hard'' function in the construction is at least as large as the iteration complexity.

In fact, the non-convex and smooth setting shares interesting parallels with the convex and non-smooth setting, despite their apparent differences (in the former setting, we seek an $\varepsilon$-stationary point, whereas in the latter setting, we seek an $\varepsilon$-minimizer).
Namely, in both settings the optimal oracle complexity is $\Theta(1/\varepsilon^2)$ in high dimension, and the optimal algorithm is (sub)gradient descent (as opposed to the convex smooth setting, for which accelerated gradient methods outperform GD).
However, for the convex non-smooth setting, we know that the large dimensionality $d\ge \Omega(1/\varepsilon^2)$ of the lower bound construction is almost necessary, because of the existence of cutting-plane methods~\citep[see, e.g.,][]{bubeck2015convex, nesterov2018cvxopt} which achieve a better complexity of $O(d\log(1/\varepsilon))$ in dimension $d \le \widetilde O(1/\varepsilon^2)$.
This raises the question of whether or not there exist analogues of cutting-plane methods for \emph{non-convex} optimization.

A negative answer to this question would substantially improve our understanding of non-convex optimization, as it would point towards fundamental algorithmic obstructions.
As such, the low-dimensional complexity of finding stationary points for non-convex optimization was investigated in a series of works~\citep{vavavis1993complexity, hinder2018cuttingplane, bubmik2020gradientflow}.
These results show the existence of algorithms which improve upon GD in dimension $d \le O(\log(1/\varepsilon))$.
This suggests that GD is actually optimal for all $d \geq \Omega(\log(1/\varepsilon))$.
To date, there has been little progress on this tantalizing conjecture because the existing low-dimensional lower bounds are delicate, relying on the theory of unpredictable random walks~\citep{vavavis1993complexity, benpemper1998unpredictable, bubmik2020gradientflow}.

\begin{table}[h!]
    \centering
    \begin{tabular}{ccccc}
        \textbf{Algorithm Class} & \textbf{Oracle} & \textbf{Complexity} & \textbf{Lower Bound} & \textbf{Upper Bound} \\
        \midrule
        Deterministic & $1^{\rm st}$ & $\Theta(1/\varepsilon^2)$ & Theorem~\ref{thm:det_lower} & GD (well-known) \\
        Randomized & $1^{\rm st}$ & $\Theta(1/\varepsilon)$ & Theorem~\ref{thm:randomized_lower} & Theorem~\ref{thm:randomized_upper} \\
        Deterministic & $0^{\rm th} + 1^{\rm st}$ & $\Theta(\log(1/\varepsilon))$ & Theorem~\ref{thm:rand_lower_0} & Theorem~\ref{thm:det_upper_0} \\
        Randomized & $0^{\rm th} + 1^{\rm st}$ & $\Theta(\log(1/\varepsilon))$ & Theorem~\ref{thm:rand_lower_0} & Theorem~\ref{thm:det_upper_0}
    \end{tabular}
    \caption{Summary of the results of this work.}\label{table:results}
\end{table}

\paragraph{Our contributions.}
In this paper, we study the task of finding an $\varepsilon$-stationary point of a smooth and univariate function $f : \R\to\R$.
Our results, which are summarized as Table~\ref{table:results}, provide a complete characterization of the oracle complexity of this task in four settings, based on whether or not the algorithm is allowed to use external randomness and whether or not the oracle outputs zeroth-order information. In particular, our lower bounds, which hold in dimension one, also hold in every dimension $d\ge 1$.
In spite of the simplicity of the setting, we can draw a number of interesting conclusions from the results.
\begin{itemize}
    \item \textbf{Optimality of GD for any dimension $d \geq 1$.} Our results imply that, among algorithms which are deterministic and only use first-order queries, GD is optimal in every dimension $d \geq 1$. This was previously known only for $d\ge \Omega(1/\varepsilon^2)$~\citep{carmonetal2020stationarypt}.
    \item \textbf{Separations between algorithm classes and oracles}. Our results exhibit a natural setting in which both randomization and zeroth-order queries provably improve the query complexity of optimization.
        It shows, in particular, that at least one of these additional ingredients is \emph{necessary} to improve upon the basic GD algorithm.
    \item \textbf{Finding stationary points for unconstrained optimization}. The methods of~\citet{vavavis1993complexity, bubmik2020gradientflow} for improving upon the complexity of GD in low dimension are applicable to the constrained case in which the domain of $f$ is the cube ${[0,1]}^d$, and it is not obvious that they can be applied to unbounded domains.
        We address this question by characterizing the oracle complexity for the unconstrained case.
\end{itemize}

\paragraph{Related works.}
Usually, optimization lower bounds are established for specific classes of algorithms, such as algorithms for which each iterate lies in the span of the previous iterates and gradients~\citep{nesterov2018cvxopt}.
As noted in~\cite{woosre2017randomizedfirstorder}, lower bounds against arbitrary randomized algorithms for convex optimization are trickier and are often loose with regards to the dimension in which the construction is embedded.
The complexity of finding stationary points is further studied in~\cite{caretal2021statptii}.

\paragraph{Conventions and notation.}
A function $f : \R^d\to\R$ is \emph{$\beta$-smooth} if it is continuously differentiable and its gradient $\nabla f$ is $\beta$-Lipschitz. If $d=1$, we shall write $f'$ instead of $\nabla f$.
We use the standard asymptotic notation $\Omega(\cdot)$, $O(\cdot)$, and $\Theta(\cdot)$.

\section{Results}

In this section, we give detailed statements of our results as well as proof sketches.
The full proofs are deferred to the appendix.
We also record the following lemma, which allows us to reduce to the case of $\beta = \Delta = 1$.

\begin{lemma}\label{lem:reduction}
    Let $\ms C_*(\varepsilon; \beta,\Delta,d, \ms O) \geq 0$ denote the complexity of finding an $\varepsilon$-stationary point over the class of $\beta$-smooth functions $f : \R^d\to\R$ with $f(0) - \inf f \le \Delta$ using an oracle $\ms O$, where given $x\in\R^d$ the oracle $\ms O$ returns either $\nabla f(x)$ (first-order information) or $(f(x), \nabla f(x))$ (zeroth- and first-order information).
    Here, $* \in \{\text{det}, \text{rand}\}$ is a subscript denoting whether or not the algorithm is allowed to use external randomness; when $* = \text{rand}$, the randomized complexity refers to the minimum number of queries required to find an $\varepsilon$-stationary point with probability at least $1/2$.
    Then, for any $\beta,\Delta,\varepsilon > 0$,
    \begin{align*}
        \ms C_*(\varepsilon; \beta,\Delta,d,\ms O)
        &= \ms C_*\bigl( \frac{\varepsilon}{\sqrt{\beta\Delta}}; 1,1,d, \ms O\bigr)\,.
    \end{align*}
\end{lemma}
\begin{proof}
    Given a $\beta$-smooth function $f : \R^d\to\R$ with $f(0) -\inf f \le \Delta$, define $g : \R^d\to\R$ via $g(x) \deq \Delta^{-1} f(\sqrt{\Delta/\beta} \, x)$.
    Then, $g$ is $1$-smooth with $g(0) - \inf g \le 1$, and it is clear that the oracle for $g$ can be simulated using the oracle for $f$.
    Moreover, an $\varepsilon/\sqrt{\beta \Delta}$-stationary point for $g$ translates into an $\varepsilon$-stationary point for $f$.
    Obviously, the reduction is reversible.
\end{proof}

Often, we will assume without loss of generality that $f(0) = 1$ and $\beta = \Delta = 1$, so that $f\ge 0$.
Also, we may assume that $f'(0) \le -\varepsilon$, since if $f'(0) \in (-\varepsilon,\varepsilon)$ then $0$ is an $\varepsilon$-stationary point of $f$, and if $f'(0) \ge \varepsilon$ we can replace $f$ by $x\mapsto f(-x)$.
We abbreviate $\ms C_*(\varepsilon; \ms O) \deq \ms C_*(\varepsilon; 1,1,1,\ms O)$, and from now on we consider $d=1$.

Let $\ms O^{1^{\rm st}}$ denote the oracle which returns first-order information (given $x\in\R$, it outputs $f'(x)$), and let $\ms O^{0^{\rm th}+1^{\rm st}}$ denote the oracle which returns zeroth- and first-order information (given $x\in\R$, it outputs $(f(x), f'(x))$). We remark that in the one-dimensional setting, we could instead assume access to an oracle $\ms O^{0^{\rm th}}$ which only outputs zeroth-order information, rather than $\ms O^{0^{\rm th}+1^{\rm st}}$; this is because we can simulate $\ms O^{1^{\rm st}}$ to arbitrary accuracy given $\ms O^{0^{\rm th}}$ with only a constant factor overhead in the number of oracle queries by using finite differences. For simplicity, we work with $\ms O^{0^{\rm th}+1^{\rm st}}$ and we will not consider $\ms O^{0^{\rm th}}$ further.

\subsection{Lower bound for randomized algorithms}\label{scn:randomized_lower}

We begin with a lower bound construction for randomized algorithms which only use first-order queries.
For simplicity, assume that $1/\varepsilon$ is an integer.
We construct a family of functions ${(f_j)}_{j\in [1/\varepsilon]}$, with the following properties.
On the negative half-line $\R_-$, each $f_j$ decreases with slope $-\varepsilon$, with $f_j(0) = 1$.
We also set the slope of $f_j$ on the positive half-line $\R_+$ to be $-\varepsilon$, but this entails that $f_j(x) < 0$ for $x > 1/\varepsilon$, violating the constraint $f_j(0) - \inf f \le 1$.
Instead, on the interval $[j-1,j]$, we modify $f_j$ to increase as much as possible while remaining $O(1)$-smooth, so that $f_j(1/\varepsilon) = f_j(0) = 1$; we can then periodically extend $f_j$ on the rest of $\R_+$.

Due to the periodicity of the construction, we can restrict our attention to the interval $[0,1/\varepsilon]$.
Without prior knowledge of the index $j$, any algorithm only has a ``probability'' (made precise in Appendix~\ref{scn:pf_randomized_lower}) of at most $\varepsilon$ of finding the interval $[j-1,j]$, which contains all of the $\varepsilon$-stationary points in $[0,1/\varepsilon]$.
Hence, we expect that any randomized algorithm must require at least $\Omega(1/\varepsilon)$ queries to find an $\varepsilon$-stationary point of $f_j$.

To make this formal, let $\Phi : [0,1] \to\R$ be a smooth function such that $\Phi(0) = 0$, $\Phi(1) = 1$, and $\Phi'(0) = \Phi'(1) = -\varepsilon$.
For example, we can take
\begin{align*}
    \Phi(x) = \begin{cases}
        2\,(1+\varepsilon) \, x^2 - \varepsilon\, x\,, & x \in [0, \frac{1}{2}]\,, \\
        2\,\Phi(\frac{1}{2}) - \Phi(1-x)\,, & x \in [\frac{1}{2}, 1]\,.
    \end{cases}
\end{align*}
We can check that $\Phi$ satisfies the desired properties and that $\Phi$ is $\beta$-smooth with $\beta = 4 \, (1+\varepsilon) \le 5$ for $\varepsilon \le \frac{1}{4}$.
Then, let
\begin{align*}
    f_j(x)
    &\deq \begin{cases}
        1-\varepsilon\, x\,, & x \in \openclose{-\infty, j-1}\,, \\
        1-\varepsilon\, (j-1) + (1-\varepsilon)\, \Phi(x-(j-1)) \,, & x \in [j-1, j]\,, \\
        f_j(j) - \varepsilon \, (x - j)\,, & x \in [j, 1/\varepsilon]\,, \\
        f_j(x-1/\varepsilon)\,, & x \in \closeopen{1/\varepsilon,\infty}\,.
    \end{cases}
\end{align*}
It follows that $f_j$ is also $5$-smooth, with $f_j(0) - \inf f_j \le 1$; see Figure~\ref{fig:randomized_lower}.

\begin{figure}[h]
    \centering
    \includegraphics[width=0.3\textwidth]{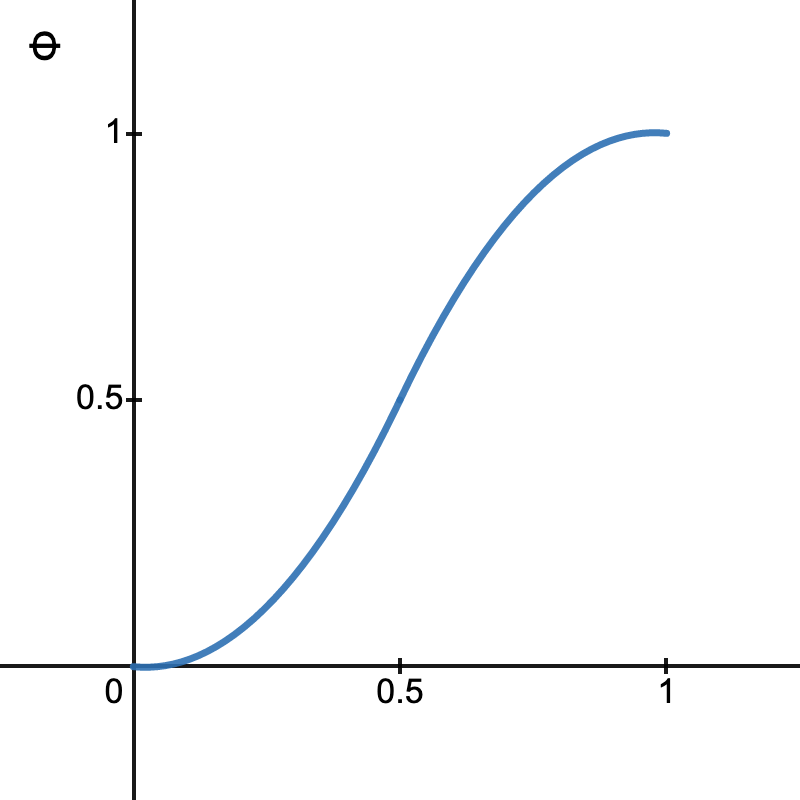}
    \hspace{5em}
    \includegraphics[width=0.3\textwidth]{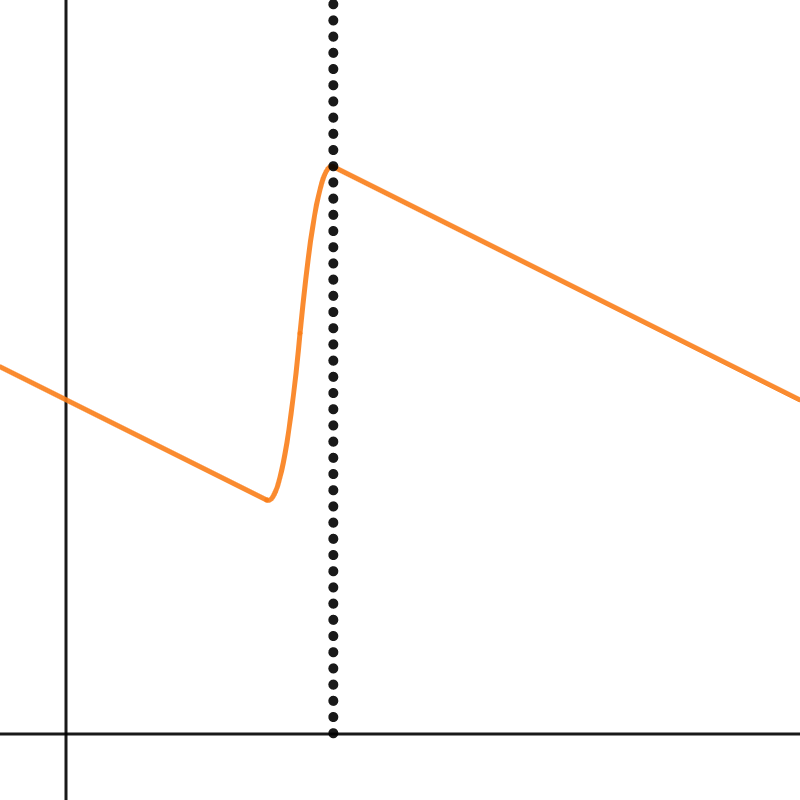}
    \caption{(Left) A plot of $\Phi$. (Right) A plot of $f_j$, where the dotted line indicates the value of $j$.}\label{fig:randomized_lower}
\end{figure}

We prove the following theorem in Appendix~\ref{scn:pf_randomized_lower}.

\begin{theorem}\label{thm:randomized_lower}
    For all $\varepsilon \in (0,\frac{1}{8})$, it holds that
    \begin{align*}
        \ms C_{\rm rand}\bigl(\varepsilon; \ms O^{1^{\rm st}}\bigr) \ge \Omega\Bigl(\frac{1}{\varepsilon}\Bigr)\,.
    \end{align*}
\end{theorem}

\subsection{An optimal randomized algorithm}

The lower bound construction of the previous section suggests a simple strategy for computing an $\varepsilon$-stationary point of $f$: namely, just repeatedly pick points uniformly at random in the interval $[0,1/\varepsilon]$.
We now show that such a strategy (together with some additional processing steps) succeeds at obtaining an $\varepsilon$-stationary point in $O(1/\varepsilon)$ queries.

\begin{minipage}[t]{0.43\textwidth}
    \centering
    \null{}
    \begin{algorithm}[H]
        \caption{\textsc{RandomSearch}}\label{alg:randomized_upper}
        \KwData{oracle $\ms O^{1^{\rm st}}$ for $f$}
\KwResult{$\varepsilon$-stationary point $x$}
\While{true}{draw $x \sim \unif([0, 2/\varepsilon])$\;

    \uIf{$\abs{f'(x)} < \varepsilon$\;}{output $x$\;}
    \uElseIf{$f'(x) > 0$\;}{call \textsc{BinarySearch}$(\ms O^{1^{\rm st}}, 0, x)$\;}
}
\end{algorithm}
\end{minipage}
\hfill{}
\begin{minipage}[t]{0.47\textwidth}
    \centering
    \null{}
    \begin{algorithm}[H]
        \caption{\textsc{BinarySearch}}\label{alg:binary_search}
        \KwData{oracle $\ms O^{1^{\rm st}}$ for $f$; initial points $x_0 < x_1$ with $f'(x_0) \le -\varepsilon$ and $f'(x_1) > 0$}
\KwResult{$\varepsilon$-stationary point $x$}
set $m \gets \frac{x_0+x_1}{2}$\;

\uIf{$\abs{f'(m)} < \varepsilon$}{output $m$\;}
\uElseIf{$f'(m) \le -\varepsilon$}{call \textsc{BinarySearch}$(\ms O^{1^{\rm st}}, m, x_1)$}
\uElseIf{$f'(m) > 0$}{call \textsc{BinarySearch}$(\ms O^{1^{\rm st}}, x_0, m)$}
\end{algorithm}
\end{minipage}

\bigskip{}

The pseudocode for the algorithms is given as Algorithms~\ref{alg:randomized_upper} and~\ref{alg:binary_search}. In short, \textsc{RandomSearch} (Algorithm~\ref{alg:randomized_upper}) uses $O(1/\varepsilon)$ queries to find a ``good point'', i.e., either an $\varepsilon$-stationary point or a point $x$ with $f'(x) > 0$. In the latter case, \textsc{BinarySearch} (Algorithm~\ref{alg:binary_search}) then locates an $\varepsilon$-stationary point using an additional $O(\log(1/\varepsilon))$ queries.

We prove the following theorem in Appendix~\ref{scn:pf_randomized_upper}.

\begin{theorem}\label{thm:randomized_upper}
    Assume that $f : \R\to\R$ is $1$-smooth, $f \ge 0$, $f(0) = 1$, and $f'(0) \le -\varepsilon$.
    Then, \textsc{RandomSearch} (Algorithm~\ref{alg:randomized_upper}) terminates with an $\varepsilon$-stationary point for $f$ using at most $O(1/\varepsilon)$ queries to the oracle with probability at least $1/2$.
\end{theorem}

As usual, the success probability can be boosted by rerunning the algorithm. In Figure~\ref{fig:exp}, we demonstrate the performance of \textsc{RandomSearch} in a numerical experiment as a sanity check.

\begin{figure}[h]
    \centering
    \includegraphics[width=0.7\textwidth]{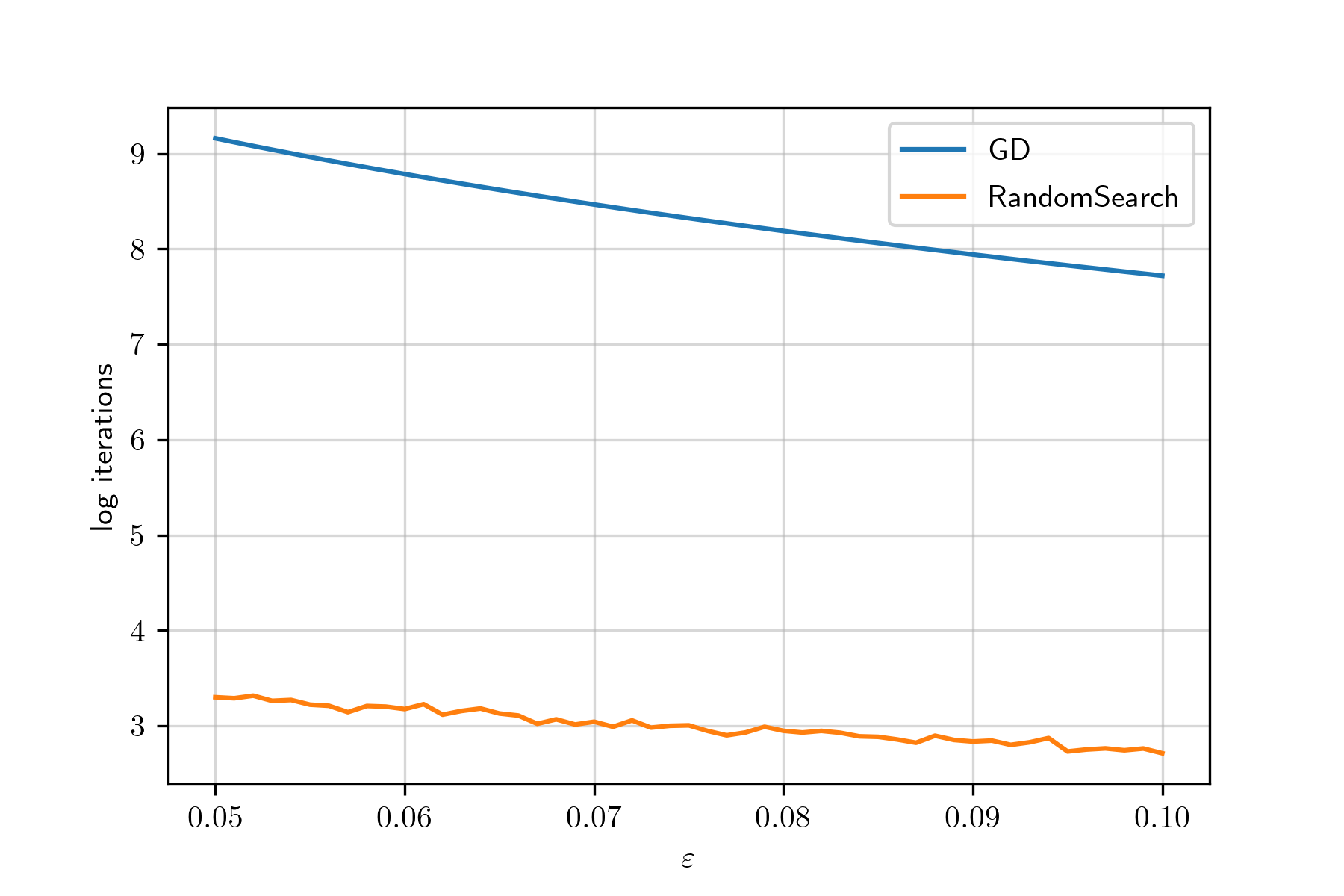}
    \caption{Iteration complexity of gradient descent (GD) vs.\ one run of \textsc{RandomSearch} (Algorithm~\ref{alg:randomized_upper}) for various choices of $\varepsilon$ on an instance of the construction in Section~\ref{scn:randomized_lower}. The flatter slope of the orange line reflects the improved $O(1/\varepsilon)$ complexity of \textsc{RandomSearch} over the $O(1/\varepsilon^2)$ complexity of GD.}\label{fig:exp}
\end{figure}

\subsection{Lower bound for deterministic algorithms}

Against the class of deterministic algorithms, the construction of Theorem~\ref{thm:randomized_lower} can be strengthened to yield a $\Omega(1/\varepsilon^2)$ lower bound.
The idea is based on the concept of a \emph{resisting oracle} $\ms O^{\rm resist}$ from~\citet{nesterov2018cvxopt} which, regardless of the query point $x$, outputs ``$f'(x) = -\varepsilon$''.
The goal then is to show that for any deterministic sequence of queries $x_1,\dotsc,x_N$, if $N \le O(1/\varepsilon^2)$, there exists a $1$-smooth function $f :\R\to\R$ with $f(0) - \inf f\le \Delta$ which is consistent with the output of the oracle, i.e., satisfies $f'(x_i) = -\varepsilon$ for all $i\in [N]$.
Note that this strategy necessarily only provides a lower bound against deterministic algorithms.\footnote{In more detail, the argument is as follows. Let $x_1,\dotsc,x_N$ be the sequence of query points generated by the algorithm when run with $\ms O^{\rm resist}$, and suppose we can find a function $f$ which is consistent with the responses of $\ms O^{\rm resist}$. Then, for a deterministic algorithm, we can be sure that \emph{had the algorithm been run with the oracle $\ms O^{1^{\rm st}}$ for $f$}, it would have generated the same sequence of query points $x_1,\dotsc,x_N$, and hence would have never found an $\varepsilon$-stationary point of $f$ among the $N$ query points. This argument fails if the algorithm incorporates external randomness.}

For simplicity of notation, since the order of the queries does not matter here, we assume that the queries are sorted: $x_1 < \cdots < x_N$.
The function $f$ that we construct has slope $-\varepsilon$ at the query points, but rapidly rises in between the query points to ensure that the condition $f(0) - \inf f \le 1$ holds.
Moreover, we will ensure that $f'(x) = -\varepsilon$ for $x\le 0$ and that $f'$ is periodic on $\R_+$ with period $1/\varepsilon$; hence, we may assume that all of the queries lie in the informative interval $(0,1/\varepsilon)$.
The key here is that for deterministic algorithms, the intervals on which the function $f$ rises can be adapted to the query points, rather than being selected in advance.

The intuition is as follows. If the algorithm has made fewer than $O(1/\varepsilon^2)$ queries, then there must be $\Omega(1/\varepsilon^2)$ disjoint intervals in $[0,1/\varepsilon]$ of length at least $\Omega(\varepsilon)$ in which there are no query points. On each such interval, we can grow our function value by $\Omega(\varepsilon^2)$ while staying smooth and with slope $-\varepsilon$ at the start and end of the interval. Hence, we can guarantee that the constructed function $f$ remains above $f(0) - 1$, while answering $f'(x) = -\varepsilon$ at every query point $x$.

To make this precise, let $\ell_i \deq x_{i+1} - x_i$ and define the function
\begin{align*}
    \Phi_i(x)
    &\deq -\varepsilon \, (x-x_i) \\
    &\qquad{} + \begin{cases} \frac{1}{2} \, {(x-x_i)}^2\,, & x \in [x_i, x_i + \frac{\ell_i}{2}]\,, \\[0.5em] \frac{\ell_i^2}{8} + \frac{\ell_i}{2} \, (x - x_i - \frac{\ell_i}{2}) - \frac{1}{2} \, {(x-x_i - \frac{\ell_i}{2})}^2\,, & x \in [x_i + \frac{\ell_i}{2}, x_{i+1}]\,. \end{cases}
\end{align*}
The construction of $\Phi_i$ satisfies the following properties:
\begin{enumerate}
    \item $\Phi_i$ is continuously differentiable and $1$-smooth on $[x_i, x_{i+1}]$.
    \item $\Phi_i(x_i) = 0$ and $\Phi_i(x_{i+1}) = \ell_i \, (\frac{\ell_i}{4} - \varepsilon)$.
    \item $\Phi_i'(x_i) = \Phi_i'(x_{i+1}) = -\varepsilon$.
\end{enumerate}
Write $x_0 \deq 0$ and $x_{N+1} \deq 1/\varepsilon$.
Recall that $x_i \in (0, 1/\varepsilon)$, for all $i\in [N]$.
We now define
\begin{align*}
    f(x) \deq
    \begin{cases}
        1-\varepsilon\, x\,, & x \in \openclose{-\infty, 0}\,, \\
        f(x_i) - \varepsilon \, (x-x_i)\,, & x \in [x_i, x_{i+1}]~\text{and}~\ell_i < 8\varepsilon \;\; (0 \le i \le N)\,, \\
        f(x_i) + \Phi_i(x)\,, & x \in [x_i, x_{i+1}]~\text{and}~\ell_i \ge 8\varepsilon \;\; (0 \le i \le N)\,, \\
        f(x-1/\varepsilon) + a\,, & x \in \closeopen{1/\varepsilon, \infty}\,,
    \end{cases}
\end{align*}
where $a \deq f(1/\varepsilon) - f(0)$.
See Figure~\ref{fig:det_lower} for an illustration of $f$.
We shall prove that when $N \le O(1/\varepsilon^2)$, then the function $f$ is $1$-smooth and satisfies $f(0) - \inf f \le 1$, thus completing the resisting oracle construction. 
It yields the following theorem, which we prove in Appendix~\ref{scn:pf_det_lower}.

\begin{figure}[h]
    \centering
    \includegraphics[width=0.4\textwidth]{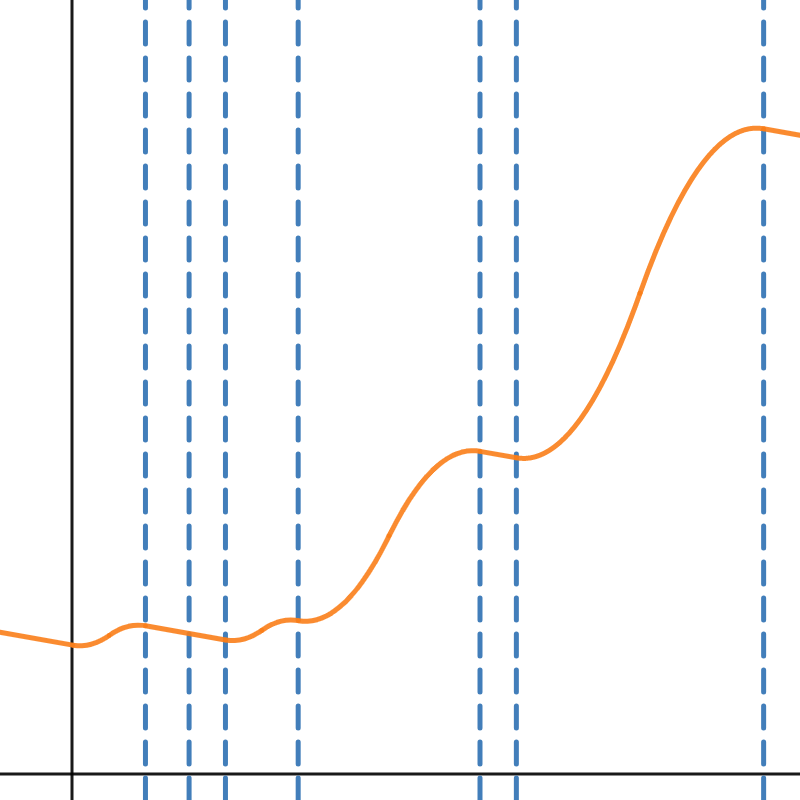}
    \caption{We plot an example of the function $f$. The dashed lines indicate the query points made by the algorithms.}\label{fig:det_lower}
\end{figure}

\begin{theorem}\label{thm:det_lower}
    For all $\varepsilon \in (0, 1)$, it holds that
    \begin{align*}
        \ms C_{\det}\bigl(\varepsilon; \ms O^{1^{\rm st}}\bigr) \ge \Omega\Bigl( \frac{1}{\varepsilon^2}\Bigr)\,.
    \end{align*}
\end{theorem}

The lower bound is matched by gradient descent.
For the sake of completeness, we provide a proof of the matching $O(1/\varepsilon^2)$ upper bound via gradient descent as Theorem~\ref{thm:gd} in Appendix~\ref{scn:preliminaries}.

\subsection{Lower bound for randomized algorithms with zeroth-order information}

We now turn towards algorithms which use the $0^{\rm th}+1^{\rm st}$-order oracle $\ms O^{0^{\rm th}+1^{\rm st}}$.
For the lower bound, we again use the family of functions ${(f_j)}_{j\in [1/\varepsilon]}$ introduced in Section~\ref{scn:randomized_lower}.
The main difference is that given a query point $x \in [0,1/\varepsilon]$, the value of $f_j(x)$ reveals whether or not the interval $[j-1,j]$ lies to the left of $x$ and hence allows for binary search to determine $j$.
Consequently, the lower bound is only of order $\Omega(\log(1/\varepsilon))$.

We prove the following theorem in Appendix~\ref{scn:pf_rand_lower_0}.

\begin{theorem}\label{thm:rand_lower_0}
    For all $\varepsilon \in (0,\frac{1}{8})$, it holds that
    \begin{align*}
        \ms C_{\det}\bigl(\varepsilon; \ms O^{0^{\rm th}+1^{\rm st}}\bigr)
        &\ge \ms C_{\rm rand}\bigl(\varepsilon; \ms O^{0^{\rm th}+1^{\rm st}}\bigr)
        \ge \Omega\Bigl(\log \frac{1}{\varepsilon}\Bigr)\,.
    \end{align*}
\end{theorem}

\subsection{An optimal deterministic algorithm with zeroth-order information}

Finally, we provide a deterministic algorithm whose complexity matches the lower bound in Theorem~\ref{thm:rand_lower_0}.
At a high level, the idea is to use the zeroth-order information to perform binary search, but the actual algorithm is slightly more involved and requires the consideration of various cases.

We summarize the idea behind the algorithm. First, as described earlier, we may freely assume $f\ge 0$, $f(0) = 1$, and $f'(0) \le -\varepsilon$. Also, we recall that if the algorithm ever sees a point $x$ with either $\abs{f'(x)} < \varepsilon$ or $f'(x) > 0$, then we are done (in the latter case, we can call Algorithm~\ref{alg:binary_search}: \textsc{BinarySearch}).
\begin{enumerate}
    \item \textsc{DecreaseGap} (Algorithm~\ref{alg:decr_gap}) checks the value of $f(2/\varepsilon)$. If $f(2/\varepsilon) \le \frac{3}{4} \, f(0)$, then we have made progress on the objective gap and we may treat $2/\varepsilon$ as the new origin. This can happen at most $O(\log(1/\varepsilon))$ times. Otherwise, we have $f(2/\varepsilon) \ge \frac{3}{4} \, f(0)$, and we move on to the next phase of the algorithm.
    \item Set $x_- \deq 0$ and $x_+ \deq 2/\varepsilon$. There are two cases: either $\frac{3}{4} \, f(x_-) \le f(x_+) \le f(x_-)$, in which case $f(x_-) - f(x_+) \le \frac{\varepsilon}{4} \, (x_+ - x_-)$, or $f(x_+) \ge f(x_-)$.
    \item The first case is handled by \textsc{BinarySearchII} (Algorithm~\ref{alg:binary_search_ii}). A simple calculation reveals that the condition $0 \le f(x_-) - f(x_+) \le \frac{3}{4} \, (x_+ - x_-)$ together with $f'(x_-) \le -\varepsilon$ implies the existence of an $\varepsilon$-stationary point in $[x_-, x_+]$. We now check the midpoint $m$ of $x_-$ and $x_+$.
    If $f(m) \notin [f(x_+), f(x_-)]$, then we arrive at the second case. Otherwise, we replace either $x_-$ or $x_+$ with $m$; one of these two choices will cut the value of $f(x_-) - f(x_+)$ by at least half, thereby ensuring that the condition $0 \le f(x_-) - f(x_+) \le \frac{3}{4} \, (x_+ - x_-)$ continues to hold. This can happen at most $O(\log(1/\varepsilon))$ times.
    \item Finally, the second case is handled by \textsc{BinarySearchIII} (Algorithm~\ref{alg:binary_search_iii}). In this case, $f(x_+) \ge f(x_-)$ together with $f'(x_-) \le -\varepsilon$ ensures that there is a stationary point in $[x_-, x_+]$.
    We then check the value of $f(m)$ where $m$ is the midpoint of $x_-$ and $x_+$. It is straightforward to check that we can replace either $x_-$ or $x_+$ with $m$ and preserve the condition $f(x_+) \ge f(x_-)$. This can happen at most $O(\log(1/\varepsilon))$ times.
\end{enumerate}

\begin{adjustwidth}{\textwidth}{}
    \centering
    \begin{algorithm}[H]
        \caption{\textsc{ZerothOrder}}\label{alg:0}
        \KwData{oracle $\ms O^{0^{\rm th}+1^{\rm st}}$ for $f$}
\KwResult{$\varepsilon$-stationary point $x$}
set $x_- \gets \textsc{DecreaseGap}(\ms O^{0^{\rm th}+1^{\rm st}}, 0)$\;

set $x_+ \gets x_- + 2/\varepsilon$\;

\uIf{$\abs{f'(x_-)} < \varepsilon$}{output $x_-$}
\uElseIf{$f(x_+) \le f(x_-)$}{call \textsc{BinarySearchII}$(\ms O^{0^{\rm th}+1^{\rm st}}, x_-, x_+)$}
\uElseIf{$f(x_+) > f(x_-)$}{call \textsc{BinarySearchIII}$(\ms O^{0^{\rm th}+1^{\rm st}}, x_-, x_+)$}
    \end{algorithm}

    \begin{algorithm}[H]
        \caption{\textsc{DecreaseGap}}\label{alg:decr_gap}
        \KwData{oracle $\ms O^{0^{\rm th}+1^{\rm st}}$ for $f$; point $x_0$}
        \KwResult{either an $\varepsilon$-stationary point $x$ or a point $x$ such that $f(x) \le f(x_0)$, $f'(x) \le -\varepsilon$, and $f(x+2/\varepsilon) \ge \frac{3}{4} \, f(x)$}
        \uIf{$\abs{f'(x_0+2/\varepsilon)} < \varepsilon$}{output $x_0+2/\varepsilon$}
        \uElseIf{$f'(x_0+2/\varepsilon) > 0$}{call \textsc{BinarySearch}$(\ms O^{0^{\rm th}+1^{\rm st}}, x_0, x_0+2/\varepsilon)$}
        \uElseIf{$f(x_0+2/\varepsilon) \ge \frac{3}{4} \, f(x_0)$}{output $x_0$\;}
        \uElse{call \textsc{DecreaseGap}$(\ms O^{0^{\rm th}+1^{\rm st}}, x_0+2/\varepsilon)$\;}
    \end{algorithm}

    \begin{algorithm}[H]
        \caption{\textsc{BinarySearchII}}\label{alg:binary_search_ii}
        \KwData{oracle $\ms O^{0^{\rm th}+1^{\rm st}}$ for $f$; points $x_- < x_+$ with $f'(x_-) \le -\varepsilon$ and $0 \le f(x_-) - f(x_+) \le \frac{\varepsilon}{4} \, (x_+ - x_-)$}
        \KwResult{an $\varepsilon$-stationary point $x$\;}
        set $m \gets \frac{x_- + x_+}{2}$\;

        \uIf{$\abs{f'(m)} < \varepsilon$}{output $m$}
        \uElseIf{$f'(m) > 0$}{call \textsc{BinarySearch}$(\ms O^{0^{\rm th}+1^{\rm st}}, x_-, m)$}
        \uElseIf{$f(m) \ge f(x_-)$}{call \textsc{BinarySearchIII}$(\ms O^{0^{\rm th}+1^{\rm st}}, x_-, m)$}
        \uElseIf{$f(m) \le f(x_+)$}{call \textsc{BinarySearchIII}$(\ms O^{0^{\rm th}+1^{\rm st}}, m, x_+)$}
        \uElseIf{$f(x_-) - f(m) \le \frac{1}{2} \, (f(x_-) - f(x_+))$}{call \textsc{BinarySearchII}$(\ms O^{0^{\rm th}+1^{\rm st}}, x_-, m)$}
        \uElseIf{$f(m) - f(x_+) \le \frac{1}{2} \, (f(x_-) - f(x_+))$}{call \textsc{BinarySearchII}$(\ms O^{0^{\rm th}+1^{\rm st}}, m, x_+)$}
    \end{algorithm}

    \begin{algorithm}[H]
        \caption{\textsc{BinarySearchIII}}\label{alg:binary_search_iii}
        \KwData{oracle $\ms O^{0^{\rm th}+1^{\rm st}}$ for $f$; points $x_- < x_+$ with $f'(x_-) \le -\varepsilon$ and $f(x_+) \ge f(x_-)$}
        \KwResult{an $\varepsilon$-stationary point $x$\;}
        set $m \gets \frac{x_- + x_+}{2}$\;

        \uIf{$\abs{f'(m)} < \varepsilon$}{output $m$}
        \uElseIf{$f'(m) > 0$}{call \textsc{BinarySearch}$(\ms O^{0^{\rm th}+1^{\rm st}}, x_-, m)$}
        \uElseIf{$f(m) \ge f(x_-)$}{call \textsc{BinarySearchIII}$(\ms O^{0^{\rm th}+1^{\rm st}}, x_-, m)$}
        \uElse{call \textsc{BinarySearchIII}$(\ms O^{0^{\rm th}+1^{\rm st}}, m, x_+)$}
    \end{algorithm}
\end{adjustwidth}

We prove the following theorem in Appendix~\ref{scn:pf_det_upper_0}.

\begin{theorem}\label{thm:det_upper_0}
    Assume that $f : \R\to\R$ is $1$-smooth, $f\ge 0$, $f(0) = 1$, and $f'(0) \le -\varepsilon$.
    Then, \textsc{ZerothOrder} (Algorithm~\ref{alg:0}) terminates with an $\varepsilon$-stationary point for $f$ using at most $O(\log(1/\varepsilon))$ queries to the oracle.
\end{theorem}

\section{Conclusion}

We have characterized the oracle complexity of finding an $\varepsilon$-stationary point of a smooth univariate function $f : \R\to\R$ in four natural settings of interest.
Besides providing insight into the limitations of gradient descent, our results exhibit surprising separations between the power of deterministic and randomized algorithms, and between algorithms that use zeroth-order information and algorithms (like gradient descent) which only use first-order information.

We conclude with a number of open directions for future research.
\begin{itemize}
    \item The main question motivating this work remains open, namely, for randomized algorithms using zeroth- and first-order information, \textbf{is it possible to prove a $\Omega(1/\varepsilon^2)$ complexity lower bound with a construction in dimension $d = O(\log(1/\varepsilon))$?}
        An affirmative answer to this question would likely build upon the lower bound techniques used in~\citet{vavavis1993complexity, bubmik2020gradientflow}.

        An even more ambitious goal is to fully characterize the query complexity of finding stationary points using zeroth- and first-order information in every fixed dimension $d$.
    \item Towards the above question, we also ask: \textbf{is there an analogue of gradient flow trapping~\citep{bubmik2020gradientflow} for unconstrained optimization?}
    \item We have established that among deterministic algorithms which only use first-order queries, gradient descent is optimal already in dimension one.
        Although randomized algorithms outperform GD in our setting of investigation, it is unclear to what extent randomness helps in higher dimension.
        Hence, we make the following bold conjecture: \textbf{can one prove a $\Omega(1/\varepsilon^2)$ complexity lower bound for randomized algorithms which only make first-order queries in dimension two?}
\end{itemize}

\acks{SC was supported by the Department of Defense (DoD) through the National Defense Science \& Engineering Graduate Fellowship (NDSEG) Program.}

\bibliography{ref}

\begin{thebibliography}{11}
\providecommand{\natexlab}[1]{#1}
\providecommand{\url}[1]{\texttt{#1}}
\expandafter\ifx\csname urlstyle\endcsname\relax
  \providecommand{\doi}[1]{doi: #1}\else
  \providecommand{\doi}{doi: \begingroup \urlstyle{rm}\Url}\fi

\bibitem[Benjamini et~al.(1998)Benjamini, Pemantle, and
  Peres]{benpemper1998unpredictable}
Itai Benjamini, Robin Pemantle, and Yuval Peres.
\newblock Unpredictable paths and percolation.
\newblock \emph{Ann. Probab.}, 26\penalty0 (3):\penalty0 1198--1211, 1998.

\bibitem[Bubeck and Mikulincer(2020)]{bubmik2020gradientflow}
S\'ebastien Bubeck and Dan Mikulincer.
\newblock How to trap a gradient flow.
\newblock In Jacob Abernethy and Shivani Agarwal, editors, \emph{Proceedings of
  Thirty Third Conference on Learning Theory}, volume 125 of \emph{Proceedings
  of Machine Learning Research}, pages 940--960. PMLR, 09--12 Jul 2020.

\bibitem[Bubeck(2015)]{bubeck2015convex}
Sébastien Bubeck.
\newblock Convex optimization: algorithms and complexity.
\newblock \emph{Foundations and Trends® in Machine Learning}, 8\penalty0
  (3-4):\penalty0 231--357, 2015.

\bibitem[Carmon et~al.(2020)Carmon, Duchi, Hinder, and
  Sidford]{carmonetal2020stationarypt}
Yair Carmon, John~C. Duchi, Oliver Hinder, and Aaron Sidford.
\newblock Lower bounds for finding stationary points {I}.
\newblock \emph{Math. Program.}, 184\penalty0 (1-2, Ser. A):\penalty0 71--120,
  2020.

\bibitem[Carmon et~al.(2021)Carmon, Duchi, Hinder, and
  Sidford]{caretal2021statptii}
Yair Carmon, John~C. Duchi, Oliver Hinder, and Aaron Sidford.
\newblock Lower bounds for finding stationary points {II}: first-order methods.
\newblock \emph{Math. Program.}, 185\penalty0 (1-2, Ser. A):\penalty0 315--355,
  2021.

\bibitem[Cover and Thomas(2006)]{coverthomas2006infotheory}
Thomas~M. Cover and Joy~A. Thomas.
\newblock \emph{Elements of information theory}.
\newblock Wiley-Interscience [John Wiley \& Sons], Hoboken, NJ, second edition,
  2006.

\bibitem[Hinder(2018)]{hinder2018cuttingplane}
Oliver Hinder.
\newblock Cutting plane methods can be extended into nonconvex optimization.
\newblock In Sébastien Bubeck, Vianney Perchet, and Philippe Rigollet,
  editors, \emph{Proceedings of the 31st Conference on Learning Theory},
  volume~75 of \emph{Proceedings of Machine Learning Research}, pages
  1451--1454. PMLR, 06--09 Jul 2018.

\bibitem[Nemirovsky and Yudin(1983)]{nemirovskyyudin1983complexity}
Arkadii~S. Nemirovsky and David~B. Yudin.
\newblock \emph{Problem complexity and method efficiency in optimization}.
\newblock Wiley-Interscience Series in Discrete Mathematics. John Wiley \&
  Sons, Inc., New York, 1983.
\newblock Translated from the Russian and with a preface by E. R. Dawson.

\bibitem[Nesterov(2018)]{nesterov2018cvxopt}
Yurii Nesterov.
\newblock \emph{Lectures on convex optimization}, volume 137 of \emph{Springer
  Optimization and Its Applications}.
\newblock Springer, Cham, 2018.

\bibitem[Vavasis(1993)]{vavavis1993complexity}
Stephen~A. Vavasis.
\newblock Black-box complexity of local minimization.
\newblock \emph{SIAM J. Optim.}, 3\penalty0 (1):\penalty0 60--80, 1993.

\bibitem[Woodworth and Srebro(2017)]{woosre2017randomizedfirstorder}
Blake Woodworth and Nathan Srebro.
\newblock Lower bound for randomized first order convex optimization.
\newblock \emph{arXiv e-prints}, art. arXiv:1709.03594, 2017.

\end{thebibliography}

\appendix

\section{Proofs}

\subsection{Preliminaries}\label{scn:preliminaries}

The standard approach for proving lower bounds against randomized algorithms is to reduce the task under consideration to a statistical estimation problem, for which we can bring to bear tools from information theory.
Namely, we use \emph{Fano's inequality}; we refer readers to~\citet[Chapter 2]{coverthomas2006infotheory} for background on entropy and mutual information.

\begin{theorem}[Fano's inequality]\label{thm:fano}
    Let $m$ be a positive integer and let $J \sim \unif([m])$.
    Then, for any estimator $\widehat J$ of $J$ which is measurable w.r.t.\ some data $Y$, it holds that
    \begin{align*}
        \Pr\{\widehat J \ne J\}
        &\ge 1 - \frac{I(J; Y) + \ln 2}{\ln m}\,,
    \end{align*}
    where $I$ denotes the mutual information.
\end{theorem}

For the sake of completeness, we also include a proof of the $O(1/\varepsilon^2)$ complexity bound for gradient descent.

\begin{theorem}\label{thm:gd}
    Suppose that $f : \R^d\to\R$ is $1$-smooth with $f(0) - \inf f \le 1$.
    Set $x_0 \deq 0$ and for $k\in\N$, consider the iterates of GD with step size $1$:
    \begin{align*}
        x_{k+1}
        &\deq x_k - \nabla f(x_k)\,.
    \end{align*}
    Then,
    \begin{align*}
        \min_{k=0,1,\dotsc,N-1}{\norm{\nabla f(x_k)}} \le \sqrt{\frac{2}{N}}\,.
    \end{align*}
\end{theorem}
\begin{proof}
    Due to the $1$-smoothness of $f$,
    \begin{align}\label{eq:descent_lemma}
        f(x_{k+1}) - f(x_k)
        &\le \langle \nabla f(x_k), x_{k+1}-x_k\rangle + \frac{1}{2} \, \norm{x_{k+1} - x_k}^2
        = - \frac{1}{2} \, \norm{\nabla f(x_k)}^2\,.
    \end{align}
    Rearranging this and summing,
    \begin{align*}
        \min_{k=0,1,\dotsc,N-1}{\norm{\nabla f(x_k)}^2}
        &\le \frac{1}{N} \sum_{k=0}^{N-1} \norm{\nabla f(x_k)}^2
        \le \frac{2}{N} \sum_{k=0}^{N-1} \{f(x_k) - f(x_{k+1})\} \\
        &\le \frac{2}{N} \, \{f(0) - f(x_N)\}
        \le \frac{2}{N}\,\{f(0) - \inf f\}
        \le \frac{2}{N}\,.
    \end{align*}
\end{proof}

\subsection{Proof of Theorem~\ref{thm:randomized_lower}}\label{scn:pf_randomized_lower}

\begin{proof}[Proof of Theorem~\ref{thm:randomized_lower}]
By making the value of $\varepsilon$ larger (up to a factor of $2$), we may assume that $1/\varepsilon$ is an integer.

We reduce the optimization task to a statistical estimation problem.
Let $J \sim \unif([1/\varepsilon])$.
Since the only regions in which $\abs{f_j'} < \varepsilon$ are contained in intervals of the form $k/\varepsilon + [j-1, j]$ for some $k\in\N$, then finding an $\varepsilon$-stationary point of $f_J$ implies that the algorithm can guess the value of $J$ (exactly).

On the other hand, we lower bound the number of queries required to guess the value of $J$.
Let $x_1,\dotsc,x_N$ denote the query points of the algorithm, which may also depend on an external source of randomness $U$.
Write $\ms O_{f_j}(x) = f_j'(x)$ for the output of the oracle for $f_j$ on the query $x$ (we omit the superscript $1^{\rm st}$ for brevity).
Let $\widehat J$ be any estimator of $J$ based on $\{x_i, \ms O_{f_J}(x_i) : i\in [N]\}$.
Then, by Fano's inequality (Theorem~\ref{thm:fano}),
\begin{align*}
    \Pr\{\widehat J \ne J\}
    &\ge 1 - \frac{I(\{x_i, \ms O_{f_J}(x_i) : i\in [N]\}; J) + \ln 2}{\ln(1/\varepsilon)}\,.
\end{align*}
First, suppose that the algorithm is deterministic.
This means that each $x_i$ is a deterministic function of $\{x_{i'}, \ms O_{f_J}(x_{i'}) : i' \in [i-1]\}$.
The chain rule for the mutual information implies that
\begin{align*}
    &I\bigl(\{x_i, \ms O_{f_J}(x_i) : i\in [N]\}; J\bigr) \\
    &\qquad \le \sum_{i=1}^N I\bigl(\ms O_{f_J}(x_i); J \bigm\vert \{x_{i'}, \ms O_{f_J}(x_{i'}) : i' \in [i-1]\}\bigr)\,.
\end{align*}
On the other hand, there are two possibilities for the $i$-th term in the summation. Either one of the previous queries already landed in an interval corresponding to $J$, in which case $J$ is already known and the mutual information is zero, or none of the previous queries have hit an interval corresponding to $J$.
In the latter case, conditionally on the information up to iteration $i$, $J$ is uniformly distributed on $1/\varepsilon-i$ remaining intervals, and so
\begin{align*}
    &I\bigl(\ms O_{f_J}(x_i); J \bigm\vert \{x_{i'}, \ms O_{f_J}(x_{i'}) : i' \in [i-1]\}\bigr) \\
    &\qquad \le H\bigl(\ms O_{f_J}(x_i) \bigm\vert \{x_{i'}, \ms O_{f_J}(x_{i'}) : i' \in [i-1]\}\bigr)
    = h\bigl( \frac{1}{1/\varepsilon-i}\bigr)\,,
\end{align*}
with $h$ denoting the entropy function $p\mapsto p\ln \frac{1}{p} + (1-p) \ln \frac{1}{1-p}$.
The last inequality follows because conditionally, $\ms O_{f_J}(x_i)$ can only be one of two possible values with probabilities $\frac{1}{1/\varepsilon-i}$ and $1-\frac{1}{1/\varepsilon-i}$ respectively.
If $N \le 1/(2\varepsilon)$, then
\begin{align*}
    I\bigl(\{x_i, \ms O_{f_J}(x_i) : i\in [N]\}; J\bigr)
    &\le 2\sum_{i=1}^N \frac{1}{1/\varepsilon-i} \ln\bigl( \frac{1}{\varepsilon}-i\bigr)
    \le 4N\varepsilon \ln \frac{1}{\varepsilon}\,.
\end{align*}
Hence,
\begin{align}\label{eq:lower_bd_prob}
    \Pr\{\widehat J \ne J\}
    &\ge 1 - \frac{4N\varepsilon \ln(1/\varepsilon) + \ln 2}{\ln(1/\varepsilon)}
    > \frac{1}{2}
\end{align}
provided that $\varepsilon \le \frac{1}{8}$ and $N \le O(1/\varepsilon)$ for a sufficiently small implied constant.
Although we have proven the bound~\eqref{eq:lower_bd_prob} for deterministic algorithms, the bound~\eqref{eq:lower_bd_prob} continues to hold for randomized algorithms simply by conditioning on the random seed $U$ which is independent of $J$.

We have proven that any randomized algorithm which is guaranteed to find an $\varepsilon$-stationary point of $f_J$ must use at least $N \ge \Omega(1/\varepsilon)$ queries, or
\begin{align*}
    \ms C(\varepsilon; 5,1,1,\ms O^{\rm 1st}) \ge \Omega\bigl( \frac{1}{\varepsilon}\bigr)\,.
\end{align*}
We conclude by applying the rescaling lemma (Lemma~\ref{lem:reduction}).
\end{proof}

\subsection{Proof of Theorem~\ref{thm:randomized_upper}}\label{scn:pf_randomized_upper}

First, we analyze the subroutine \textsc{BinarySearch}.

\begin{lemma}\label{lem:binary_search}
    Suppose that $f$ is $1$-smooth.
    Then, \textsc{BinarySearch} (Algorithm~\ref{alg:binary_search}) terminates with an $\varepsilon$-stationary point for $f$ using at most $O(\log \frac{x_1-x_0}{\varepsilon})$ queries to the oracle.
\end{lemma}
\begin{proof}
    Since $f$ is $1$-smooth, $f(x_0) \le -\varepsilon$ and $f(x_1) > 0$ cannot hold if $x_1 - x_0 \le \varepsilon$.
    Moreover, each time that \textsc{BinarySearch} fails to find an $\varepsilon$-stationary point for $f$, the length of the interval $[x_0, x_1]$ is cut in half.
    The result follows.
\end{proof}

We also need one lemma about continuous functions on $\R$.

\begin{lemma}\label{lem:analysis}
    Let $g : \R\to\R$ be continuous, let $I$ be a compact and non-empty interval, and let $\varepsilon > 0$.
    Then, there is a finite collection of disjoint closed intervals which cover $I \cap \{g \ge \varepsilon\}$ and which are contained in $I \cap \{g\ge 0\}$.
\end{lemma}
\begin{proof}
    For each $x \in S \deq I \cap \{g\ge \varepsilon\}$, by continuity of $g$ there exists a closed interval $I_x \subseteq I$ such that $x$ belongs to the interior of $I_x$ and such that $g \ge 0$ on $I_x$.
    The collection ${(I_x)}_{x\in S}$ covers the compact set $S$, so we can extract a finite subcover. The connected components of the union of the finite subcover consist of disjoint closed intervals.
\end{proof}

We are now ready to prove Theorem~\ref{thm:randomized_upper}.

\medskip{}

\begin{proof}[Proof of Theorem~\ref{thm:randomized_upper}]
    Let $x \sim \unif([0, 2/\varepsilon])$.
    If $\abs{f'(x)} < \varepsilon$, then we are done, and if $f'(x) > 0$, then Lemma~\ref{lem:binary_search} shows that \textsc{BinarySearch} terminates with an $\varepsilon$-stationary point of $f$ using $O(\log(1/\varepsilon))$ queries.
    What remains to show is that $x$ satisfies either $\abs{f'(x)} < \varepsilon$ or $f'(x) > 0$ with probability at least $\Omega(\varepsilon)$, which implies that Algorithm~\ref{alg:randomized_upper} succeeds using $O(1/\varepsilon)$ queries with probability at least $1/2$.

    Let $\mf m$ denote the Lebesgue measure restricted to $[0,2/\varepsilon]$.
    Then,
    \begin{align*}
        1
        &\ge f(0) - f(2/\varepsilon)
        = -\int_{[0,2/\varepsilon]} f' \\
        &\ge \varepsilon \, \mf m\{f' \le -\varepsilon\} - \varepsilon \, \mf m\{\abs{f'} < \varepsilon\} - \int_{[0,2/\varepsilon] \cap \{f' \ge \varepsilon\}} f'\,.
    \end{align*}
    From Lemma~\ref{lem:analysis}, we can cover the set $[0,2/\varepsilon] \cap \{f' \ge \varepsilon\}$ with a union of disjoint closed intervals $\bigcup_{k=1}^K I_k \subseteq [0,2/\varepsilon] \cap \{f' \ge 0\}$.
    On $I_k$, the smoothness of $f$ ensures that
    \begin{align*}
        -\int_{I_k} f'
        &\ge -\mf m(I_k) \, \underbrace{f'(\inf I_k)}_{\le \varepsilon} - \int_{I_k} \, (x - \inf I_k) \, \D x
        \ge -\varepsilon \, \mf m(I_k) - \frac{1}{2} \, {\mf m(I_k)}^2\,.
    \end{align*}
    Write $\ell_k \deq \mf m(I_k) = \sup I_k - \inf I_k$.
    Note that $\sum_{k=1}^K \ell_k \le \mf m\{f' \ge 0\}$.
    Thus,
    \begin{align*}
        -\int_{[0,2/\varepsilon] \cap \{f' \ge \varepsilon\}} f'
        &\ge -\varepsilon \sum_{k=1}^K \ell_k - \frac{1}{2} \sum_{k=1}^K \ell_k^2
        \ge -\varepsilon \sum_{k=1}^K \ell_k - \frac{1}{2} \, \Bigl( \sum_{k=1}^K \ell_k \Bigr){\vphantom{\Big|}}^2 \\
        &\ge -\varepsilon \, \mf m\{f'\ge 0\} - \frac{1}{2} \, {\mf m\{f' \ge 0\}}^2\,.
    \end{align*}
    Now suppose that $\mf m\{\abs{f'} < \varepsilon~\text{or}~f' \ge \varepsilon\} \le c_0$, where $c_0 > 0$ is a constant to be chosen later.
    In this case, the inequalities above imply
    \begin{align*}
        1 + 2c_0 \varepsilon + \frac{1}{2} \, c_0^2
        &\ge \varepsilon \,\mf m\{f' \le -\varepsilon\}
        \ge \varepsilon \, \bigl( \frac{2}{\varepsilon}-\mf m\{\abs{f'} < \varepsilon~\text{or}~f' \ge \varepsilon\}\bigr)
    \end{align*}
    which, when rearranged, yields
    \begin{align*}
        1 + 3c_0 \varepsilon + \frac{1}{2}\, c_0^2
        &\ge 2\,.
    \end{align*}
    If $c_0$ is a sufficiently small absolute constant, we arrive at a contradiction.

    We conclude that $\mf m\{\abs{f'} < \varepsilon~\text{or}~f' \ge \varepsilon\} \ge c_0$, which means that the random point $x$ will be good in the sense that either $\abs{f'(x)} < \varepsilon$ or $f'(x) \ge \varepsilon$.
    The probability that Algorithm~\ref{alg:randomized_upper} fails to obtain a good random point in $N$ tries is at most ${(1-c_0 \varepsilon/2)}^N$, which can be made at most $1/2$ by taking $N= \Theta(1/\varepsilon)$.
    We conclude that with probability at least $1/2$, using
    \begin{align*}
        O\bigl(\frac{1}{\varepsilon} + \log \frac{1}{\varepsilon}\bigr)
        = O\bigl( \frac{1}{\varepsilon}\bigr)\quad\text{queries}\,,
    \end{align*}
    Algorithm~\ref{alg:randomized_upper} finds an $\varepsilon$-stationary point.
\end{proof}

\subsection{Proof of Theorem~\ref{thm:det_lower}}\label{scn:pf_det_lower}

\begin{proof}[Proof of Theorem~\ref{thm:det_lower}]
    The goal is to show that when $N \le O(1/\varepsilon^2)$, the resisting oracle construction succeeds, and hence no deterministic algorithm can find an $\varepsilon$-stationary point of an arbitrary $1$-smooth function with objective gap at most $1$ using $N$ queries.

    For the resisting oracle construction, the crux of the matter is to show that $a = f(1/\varepsilon) - f(0) \ge 0$.
    Indeed, if this holds, then since $f$ is clearly bounded below by $0$ on $[0, 1/\varepsilon]$ it will follow that $f\ge 0$ on all of $\R$, and hence $f(0) - \inf f \le 1$.

    Let $I$ be the set of indices $i \in [N]$ for which $\ell_i \ge 8\varepsilon$.
    Since $f$ has slope $-\varepsilon$ on all of the linear pieces, then over all of the linear pieces the value of $f$ drops by at most $1$ on the interval $[0, 1/\varepsilon]$.
    The goal is to show that
    \begin{align*}
        \sum_{i\in I} \{f(x_{i+1}) - f(x_i)\}
        &\overset{!}{\ge} 1\,.
    \end{align*}
    To prove this, write
    \begin{align*}
        \frac{1}{\varepsilon}
        &= \sum_{i=1}^N \ell_i
        = \sum_{i\in I} \ell_i + \sum_{i\in I^\comp} \ell_i
        \le \sum_{i\in I} \ell_i + 8\varepsilon \,\abs{I^\comp}\,.
    \end{align*}
    There are two cases to consider.
    If $\abs{I^\comp} \ge \frac{1}{16\varepsilon^2}$ queries, then we are done, as the algorithm has made $\Omega(1/\varepsilon^2)$ queries.
    Otherwise, $\abs{I^\comp} \le \frac{1}{16\varepsilon^2}$, in which case
    \begin{align*}
        \frac{1}{2\varepsilon}
        &\le \sum_{i\in I} \ell_i\,.
    \end{align*}
    In this second case, we now have
    \begin{align*}
        \sum_{i\in I} \{f(x_{i+1}) - f(x_i)\}
        &= \sum_{i\in I} \Phi_i(x_{i+1})
        = \sum_{i\in I} \ell_i \, \bigl( \frac{\ell_i}{4} - \varepsilon\bigr)
        \ge \frac{1}{8} \sum_{i\in I} \ell_i^2 \\
        &\ge \frac{1}{8\, \abs I}\, \Bigl( \sum_{i\in I} \ell_i \Bigr){\vphantom{\Big|}}^2
        \ge \frac{1}{32\varepsilon^2 \, \abs I}\,.
    \end{align*}
    This is greater than $1$ provided $\abs I \le \frac{1}{32\varepsilon^2}$.

    In summary, the resisting oracle construction is valid provided $\abs I \le \frac{1}{32\varepsilon^2}$ and $\abs{I^\comp} \le \frac{1}{16\varepsilon^2}$.
    Since $\abs I + \abs{I^\comp} = N$, any deterministic algorithm which finds an $\varepsilon$-stationary point must use at least $N \ge \min\{\frac{1}{32\varepsilon^2}, \frac{1}{16\varepsilon^2}\} = \frac{1}{32\varepsilon^2}$ queries, or
    \begin{align*}
        \ms C_{\det}(\varepsilon; \ms O^{1^{\rm st}})
        &\ge \frac{1}{32\varepsilon^2}\,.
    \end{align*}
\end{proof}

\subsection{Proof of Theorem~\ref{thm:rand_lower_0}}\label{scn:pf_rand_lower_0}

\begin{proof}[Proof of Theorem~\ref{thm:rand_lower_0}]
    The proof is very similar to the proof of Theorem~\ref{thm:randomized_lower}.
    We follow the proof up to the point where
    \begin{align*}
        &I\bigl(\ms O_{f_J}(x_i); J \bigm\vert \{x_{i'}, \ms O_{f_J}(x_{i'}) : i' \in [i-1]\}\bigr) \\
        &\qquad \le H\bigl(\ms O_{f_J}(x_i) \bigm\vert \{x_{i'}, \ms O_{f_J}(x_{i'}) : i' \in [i-1]\}\bigr)\,,
    \end{align*}
    where now $\ms O_{f_J}(x) = \{f_J(x), f_J'(x)\}$ returns zeroth- and first-order information.
    The key point now is that since $x_i$ is deterministic (conditioned on previous queries), $\ms O_{f_J}(x_i)$ can only take a constant number of possible values, and so the above entropy term is $O(1)$ (as opposed to Theorem~\ref{thm:randomized_lower}, in which the entropy term was of order $O(\varepsilon \log(1/\varepsilon))$).
    Plugging this into Fano's inequality (Theorem~\ref{thm:fano}), we obtain
    \begin{align*}
        \Pr\{\widehat J \ne J\}
        &\ge 1 - \frac{O(N) + \ln 2}{\ln(1/\varepsilon)} > \frac{1}{2}\,,
    \end{align*}
    provided that $\varepsilon \le \frac{1}{8}$ and $N \le O(\log(1/\varepsilon))$.
    This proves that $\Omega(\log(1/\varepsilon))$ queries to $\ms O^{0^{\rm th}+1^{\rm st}}$ are necessary to find an $\varepsilon$-stationary point, even for a randomized algorithm.
\end{proof}

\subsection{Proof of Theorem~\ref{thm:det_upper_0}}\label{scn:pf_det_upper_0}

We prove the correctness of the algorithms in reverse order, beginning with \textsc{BinarySearchIII}.

\begin{lemma}\label{lem:binary_search_iii}
    Let $f : \R\to\R$ be $1$-smooth.
    Then, \textsc{BinarySearchIII} (Algorithm~\ref{alg:binary_search_iii}) terminates with an $\varepsilon$-stationary point of $f$ using $O(\log \frac{x_+ - x_-}{\varepsilon})$ queries to the oracle.
\end{lemma}
\begin{proof}
    Due to the $1$-smoothness of $f$, if $x_+ - x_- < \varepsilon$, then $f' < 0$ on the interval $[x_-, x_+]$, which contradicts the hypothesis $f(x_+) \ge f(x_-)$.
    Hence, \textsc{BinarySearchIII} can only recursively call itself at most $O(\log \frac{x_+-x_-}{\varepsilon})$ times.
    If it calls \textsc{BinarySearch}, then by Lemma~\ref{lem:binary_search} this only uses an additional $O(\log \frac{x_+ - x_-}{\varepsilon})$ queries.
\end{proof}

\begin{lemma}\label{lem:binary_search_ii}
    Let $f : \R\to\R$ be $1$-smooth.
    Then, \textsc{BinarySearchII} (Algorithm~\ref{alg:binary_search_ii}) terminates with an $\varepsilon$-stationary point of $f$ using $O(\log \frac{x_+ - x_-}{\varepsilon})$ queries to the oracle.
\end{lemma}
\begin{proof}
    First, we check that when \textsc{BinarySearchII} calls itself, the preconditions of \textsc{BinarySearchII} continue to be met.
    Suppose for instance that $0 \le f(x_-) - f(m) \le \frac{1}{2} \, (f(x_-) - f(x_+))$.
    Since $0 \le f(x_-) - f(x_+) \le \frac{\varepsilon}{4} \, (x_+ - x_-)$ by hypothesis, then
    \begin{align*}
        0
        &\le f(x_-) - f(m)
        \le \frac{\varepsilon}{8} \, (x_+ - x_-)
        = \frac{\varepsilon}{4} \, (x_- - m)\,,
    \end{align*}
    which is what we wanted to show.
    The other case is similar.

    Next, we argue that \textsc{BinarySearchII} terminates.
    The hypotheses of \textsc{BinarySearchII} imply that there is an $\varepsilon/2$-stationary point in the interval $[x_-,x_+]$.
    Indeed, if this were not the case, then $f' \le -\varepsilon/2$ on the entire interval, so $f(x_+) = f(x_-) + \int_{[x_-,x_+]} f' \le f(x_-) - \frac{\varepsilon}{2} \, (x_+ - x_-)$, but this contradicts the assumption $f(x_-) - f(x_+) \le \frac{\varepsilon}{4} \, (x_+ - x_-)$.
    Therefore, if $x_+ - x_- < \frac{\varepsilon}{2}$, it would follow that $f'(x_-) > -\varepsilon$, which contradicts the hypothesis $f'(x_-) \le -\varepsilon$.
    Since the value of $x_+ - x_-$ is cut in half each time that \textsc{BinarySearchII} calls itself, we conclude that this can happen at most $O(\log \frac{x_+ - x_-}{\varepsilon})$ times.
    If \textsc{BinarySearchII} calls either \textsc{BinarySearch} or \textsc{BinarySearchIII}, then by Lemma~\ref{lem:binary_search} and Lemma~\ref{lem:binary_search_iii}, this uses at most an additional $O(\log \frac{x_+ - x_-}{\varepsilon})$ queries to the oracle.
\end{proof}

\begin{lemma}\label{lem:decrease_gap}
    Let $f : \R\to\R$ be $1$-smooth.
    Then, \textsc{DecreaseGap} (Algorithm~\ref{alg:decr_gap}) terminates, either with an $\varepsilon$-stationary point of $f$, or with a point $x$ such that $f(x) \le f(x_0)$ and $f(x+2/\varepsilon) \ge \frac{3}{4} \, f(x)$, using $O(\log \frac{1}{\varepsilon})$ queries to the oracle.
\end{lemma}
\begin{proof}
    Each time \textsc{DecreaseGap} calls itself, the value of $f(x_0)$ decreases by a factor of $\frac{3}{4}$.
    If $f'(x_0) \le -\varepsilon$, then from~\eqref{eq:descent_lemma} we deduce that $f(x_0) \ge \frac{1}{2} \, \abs{f'(x_0)}^2 \ge \varepsilon^2/2$.
    Hence, \textsc{DecreaseGap} can call itself at most $O(\log \frac{1}{\varepsilon^2}) = O(\log \frac{1}{\varepsilon})$ times.
    If it calls \textsc{BinarySearch}, then by Lemma~\ref{lem:binary_search} this uses an additional $O(\log \frac{1}{\varepsilon})$ queries to the oracle.
\end{proof}

Finally, we are ready to verify the correctness of \textsc{ZerothOrder} (Algorithm~\ref{alg:0}).

\medskip{}

\begin{proof}[Proof of Theorem~\ref{thm:det_upper_0}]
    From Lemma~\ref{lem:decrease_gap}, if $\abs{f'(x_-)} > \varepsilon$ then we must have $f'(x_-) \le -\varepsilon$ and $f(x_+) \ge \frac{3}{4} \, f(x_-)$.
    There are two cases.
    If $f(x_+) \le f(x_-)$, then we know that
    \begin{align*}
        0 \le f(x_-) - f(x_+) \le \frac{1}{4} \, f(x_-)
        \le \frac{1}{4}
        = \frac{\varepsilon}{8} \, (x_+ - x_-)
    \end{align*}
    so the preconditions of \textsc{BinarySearchII} are met; by Lemma~\ref{lem:binary_search_ii}, \textsc{ZerothOrder} terminates with an $\varepsilon$-stationary point of $f$ using $O(\log \frac{1}{\varepsilon})$ additional queries.
    In the other case $f(x_+) \ge f(x_-)$, by Lemma~\ref{lem:binary_search_iii}, \textsc{ZerothOrder} again terminates with an $\varepsilon$-stationary point of $f$ using $O(\log \frac{1}{\varepsilon})$ additional queries.
    This concludes the proof.
\end{proof}

\end{document}